\documentclass[]{amsart}
\usepackage{hyperref}
\usepackage[utf8]{inputenc}
\usepackage{amsmath}
\usepackage{amsfonts}
\usepackage{amsthm}
\usepackage{amssymb}
\usepackage{color}
\usepackage{todonotes}
\usepackage{comment}
\usepackage{tikz}
\usepackage{hvfloat}
\usepackage{enumerate}
\usepackage{todonotes}
\usetikzlibrary{decorations.markings}
\usetikzlibrary{arrows}
\usepackage{graphicx}
\usepackage{caption}
\usepackage{subcaption}
\usepackage{mathtools}
\usepackage[all]{xy}

\usepackage{xcolor}

\makeatletter
\renewcommand{\boxed}[1]{\text{\fboxsep=.2em\fbox{\m@th$\displaystyle#1$}}}
\makeatother

\newcommand{\Mod}[1]{\ (\mathrm{mod}\ #1)}

\newcommand{\Z}{\mathbb{Z}}

\newcommand{\N}{\mathbb{N}}

\renewcommand{\le}{\leqslant}
\renewcommand{\ge}{\geqslant}
\newcommand{\quotient}[2]{{\raisebox{.2em}{$#1$}\left/\raisebox{-.2em}{$#2$}\right.}}

\theoremstyle{plain}
\newtheorem{thm}{Theorem}[section]
\newtheorem{lem}[thm]{Lemma}
\newtheorem{prop}[thm]{Proposition}

\newtheorem{quR}{Question}

\makeatletter
\newtheorem*{rep@theorem}{\rep@title}
\newcommand{\newreptheorem}[2]{%
\newenvironment{rep#1}[1]{%
\def\rep@title{#2 \ref{##1}}%
\begin{rep@theorem}}%
{\end{rep@theorem}}}
\makeatother
\newreptheorem{theorem}{Theorem}

\makeatletter
\newtheorem*{rep@prop}{\rep@title}
\newcommand{\newrepprop}[2]{%
\newenvironment{rep#1}[1]{%
\def\rep@title{#2 \ref{##1}}%
\begin{rep@prop}}%
{\end{rep@prop}}}
\makeatother
\newreptheorem{prop}{Proposition}

\newtheorem*{thm*}{Theorem}
\newtheorem*{lem*}{Lemma}
\newtheorem*{prop*}{Proposition}
\newtheorem*{cor*}{Corollary}
\newtheorem*{qu*}{Question}
\newtheorem*{dt*}{Definition and Theorem}
\newtheorem*{exmp*}{Example}
\newtheorem*{exmps*}{Examples}
\newtheorem*{dprop*}{Definition and Proposition}
\newtheorem*{conj*}{Conjecture}

\theoremstyle{definition}
\newtheorem{defn}[thm]{Definition}

\newtheorem*{defn*}{Definition}
\newtheorem*{not*}{Notation}

\theoremstyle{plain}
\newtheorem{rem}[thm]{Remark}
\newtheorem*{rem*}{Remark}

\DeclareMathOperator\FSym{FSym}
\DeclareMathOperator\Alt{Alt}
\DeclareMathOperator\FAlt{FAlt}
\DeclareMathOperator\Sym{Sym}

\DeclareMathOperator\Aut{Aut}

\DeclareMathOperator\supp{supp}

\begin{document}
\title{On the spread of infinite groups}

\author{Charles Garnet Cox}
\address{School of Mathematics, University of Bristol, Bristol BS8 1UG, UK}
\email{charles.cox@bristol.ac.uk}

\thanks{}


\date{\today}
\begin{abstract}
A group is $\frac32$\emph{-generated} if every non-trivial element is part of a generating pair. In 2019, Donoven and Harper showed that many Thompson groups are $\frac32$-generated and posed five questions. The first of these is whether there exists a 2-generated group with every proper quotient cyclic that is not $\frac32$-generated. This is a natural question given the significant work in proving that no finite group has this property, but we show that there is such an infinite group. The groups we consider are a family of finite index subgroups of the Houghton group $\FSym(\Z)\rtimes\Z$. We then show that the first two groups in our family are $\frac32$-generated, and investigate the related notion of \emph{spread} for these groups. We are able to show that they have finite spread which is greater than 2. These are therefore the first infinite groups to be shown to have finite positive spread, and the first to be shown to have spread greater than 2 (other than $\Z$ and the Tarski monsters, which have infinite spread).
\end{abstract}
\maketitle
\section{Introduction}
A group $G$ is $\frac32$\emph{-generated} if for every non-trivial $g\in G$ there exists an $h\in G$ such that $\langle g, h\rangle=G$. This name was chosen because the property is stronger than being $2$-generated, but is weaker than being cyclic. A natural generalisation of this is the notion of spread, introduced in \cite{spread}. The \emph{spread} of a group $G$, denoted by $s(G)$, is the supremum of those $k\in \N\cup\{0\}$ such that, for every $g_1, \ldots, g_k \in G\setminus\{1\}$, there exists an $h\in G$ with
$$\langle g_1, h\rangle=\ldots=\langle g_k, h\rangle=G.$$
This notion has attracted significant attention for finite groups, e.g.\ \cite{spread, spread2, nospread1, H, J, C, D}. But the first paper on spread for infinite groups, \cite{first}, only recently appeared. They showed that the Higman-Thompson and Brin-Thompson groups are $\frac32$-generated i.e.\ have spread at least 1. They also posed many interesting questions, which we investigate here. The following lemma is a good place to start.

\begin{lem*} If $G$ is $\frac32$-generated, then every proper quotient of $G$ is cyclic.
\end{lem*}
\begin{proof} Let $H\lhd G$ be chosen so that the quotient of $G$ by $H$ is not cyclic, and let $h\in H\setminus\{1\}$. Then $\langle g, h\rangle\ne G$ for every $g\in G$.
\end{proof}

In the case of finite groups there has been an exceptional effort \cite{F, J, G, I} to prove the converse of this statement, which was first conjectured in \cite{nospread1} and eventually completed in \cite{J}. For infinite groups the converse to the above lemma is false for a trivial reason: the existence of infinite simple groups that are not 2-generated. Note that there exist finitely generated, but not $2$-generated, infinite simple groups, as shown in \cite{guba}. This leads us to the first question raised in \cite{first}. 

\begin{quR} Is there a 2-generated group with no noncyclic proper quotients that is not $\frac32$-generated?
\end{quR}

The answer to this is yes, as indicated by our first theorem. The groups considered are all finite index in the second Houghton group. This is the group $\FSym(\Z)\rtimes\langle t\rangle$, where $\FSym(\Z)$ consists of all bijections of $\Z$ which move only finitely many points, and $t$ is the translation map sending $z$ to $z+1$ for all $z\in \Z$. Then $\FSym(\Z)$ contains $\Alt(\Z)$, which consists of exactly those permutations which are even. We  consider, for $k\in \N$, the groups $G_k:=\langle \Alt(\Z), t^k\rangle$. These groups were introduced in \cite{Hou2}. Their automorphism group structure was described in \cite{Cox1}.

\begin{reptheorem}{question1} The groups $G_k:=\langle \Alt(\Z), t^k\rangle$, indexed over the naturals, satisfy:
\begin{enumerate}[i)]
\item for every $k \in \N$, the group $G_k$ is 2-generated;
\item for every $k \in \N$, every proper quotient of $G_k$ is cyclic; but
\item for every $k \in \{3, 4, \ldots\}$, the group $G_k$ is not $\frac32$-generated i.e.\ $s(G_k)=0$.  
\end{enumerate}
\end{reptheorem} 
For (i), we make a careful choice of elements to generate each $G_k$. Although the groups appearing in Theorem \ref{question1} are not simple, a short argument shows that they have $\Alt(\Z)$ as a unique minimal normal subgroup. This ensures that (ii) holds. Then (iii) easily follows from $(1\;2\;3)\in G_k$  not being part of a generating pair. We then show that these groups are in some sense `not far' from having infinite spread, with the obstruction coming from the finite order elements.
 
\begin{reptheorem}{almostinfinitespread} Let $k, n \in \N$ and $g_1, \ldots, g_n \in G_k\setminus \Alt(\Z)$. Then there exists an $h\in [t^k]\subset G_k$ such that $\langle g_1, h\rangle=\ldots=\langle g_n, h\rangle=G_k$.
\end{reptheorem} 
We then investigate the other questions of \cite{first}, which require additional notions. Given a group $G$, let us say that the \emph{spread on} $X\subseteq G$ is the supremum of those $k\in \N\cup\{0\}$ such that for any $g_1, \ldots, g_k\in G$ there exists an $h \in X$ so that
\begin{equation}
\langle g_1, h\rangle=\ldots=\langle g_k, h\rangle=G.
\end{equation}

 Although the author is unaware of it having been used, the notation $s_X(G)$ appears natural in this context. If there is a set $X$ such that $s_X(G)\ge 1$, then $X$ is called a \emph{total dominating set} of $G$. We can then define the \emph{uniform spread} of $G$ as the supremum of $\{s_C(G)\;:\;C$ a conjugacy class of $G\}$, and denote this by $u(G)$. The concept of uniform spread was introduced in \cite{nospread1}; both \cite{L} and \cite{H} provide further interesting results for finite groups. Note that $u(G)\le s(G)$ by definition. 
\begin{repprop}{nofinitedomset} No finite total dominating set exists for $G_1$ or $G_2$. 
\end{repprop}
 Since all of the elements we use for $G_1$ lie in $[t]$ and the elements we use for $G_2$ lie in $[t^2]$, our computations for the spread of $G_1$ and $G_2$ provide information about the uniform spread of these groups. Note that \cite{J} also showed that no finite group $G$ has $s(G)=1$, which was open from 1975 having been discussed in \cite{spread}. Deciding whether this is also the case for infinite groups is \cite[Question 2]{first}, but we see that $G_1$ and $G_2$ do not resolve this question.
\begin{reptheorem}{spreadofG1G2}  Let $G_1=\langle (1\;2\;3), t\rangle$ and $G_2=\langle (1\;2\;3), t^2\rangle$. Then $u(G_1)\ge 2$, $u(G_2)\ge 2$, $s(G_1)\le 9$, and $s(G_2)\le 9$.

\end{reptheorem}
Note that $G_1$ and $G_2$ are the first infinite groups known to have finite positive spread, and the first to be shown to have spread greater than 2 (other than $\Z$ and the Tarski monsters, which have infinite spread). Also, given any $k\in \{2, 3, \ldots\}$, note that $G_{2k}$ is index $k$ in $G_2$ but $s(G_2)>0$ whilst $s(G_{2k})=0$.

We thank Benjamin Barrett, Scott Harper, and Jeremy Rickard for helpful conversations regarding the proof of the following lemma, which significantly narrows the search for $\frac32$-generated groups. Harry Petyt has also pointed out that the residually finite condition can be weakened (by not requiring the quotient to be finite). 
\begin{lem} \label{notrf} Let $G$ be infinite, residually-finite, and have every proper quotient cyclic. Then $G=\Z$.
\end{lem}
\begin{proof} Our proof involves three parts: that the group is abelian; that the group is finitely generated; and that it must be $\Z$.

Assume there exist $a, b \in G$ such that $[a, b]$, the commutator of $a$ and $b$, is non-trivial. Using that $G$ is residually finite, there must exist a non-trivial normal subgroup $N$ of $G$ with $[a, b]\not\in N$. But then the images of $a$ and $b$ in the quotient of $G$ by $N$ must be non-trivial, and $G/N$ is non-abelian and so not cyclic.

Now let $x \in G\setminus\{1\}$. Then $G/\langle x\rangle$ is cyclic, and so is generated by some $y$. Taking a preimage $y'$ of $y$ implies that $G=\langle x, y'\rangle$. Thus $G$ is finitely generated.

Since $G$ is a finitely generated abelian group, $G=A\times \Z^n$ for some finite abelian group $A$ and $n \in \N\cup\{0\}$. If $n=0$, then $G$ is finite. Finally, using the condition on quotients, it is clear that $n=1$ and $A$ must be trivial.
\end{proof}

\vspace{0.1cm}
\noindent\textbf{Acknowledgements.} I thank Scott Harper of the University of Bristol for his illuminating talks, ever-keen interest to discuss group theory, and comments on this work. I also thank both him and Casey Donoven for the interesting questions they raised in \cite{first}. I thank the anonymous referee for their comments. I thank François Le Maître, Université de Paris, for his improvements to Lemma \ref{lemG1ub}. Finally, I thank Tim Burness of the University of Bristol for his support and encouragement.

\vspace{0.3cm}
\noindent\textbf{Competing interests.} The author is employed by the University of Bristol.

\section{The Houghton groups}
For each $n\in \N$, the Houghton group $H_n$ acts on the set $X_n:=\{1, \ldots, n\}\times \N$ in a way that is `eventually a translation'. This paper deals exclusively with subgroups of $H_2$, and so it is this group - and a particularly nice representation of it - that we introduce here. For a more in depth introduction see, for example, \cite{Cox1}.

\begin{defn} Given $X\ne \emptyset$, let $\Sym(X)$ denote the group of all bijections on $X$. For any $g\in \Sym(X)$ we then define $\supp(g):=\{x\in X : (x)g\ne x\}$, which is often called the \emph{support} of $g$. Let $\FSym(X)$ consist of those $g\in \Sym(X)$ with $|\supp(g)|<\infty$, and let $\Alt(X)\le\FSym(X)$ consist of all even permutations.
\end{defn}
Note that, if $|X|=n$, then $\Sym(X)=\FSym(X)\cong S_n$ and $\Alt(X)\cong A_n$.
\begin{not*} Let $t\in \Sym(\Z)$ be defined as $t: z\mapsto z+1$ for all $z \in \Z$.
\end{not*}
\begin{defn} The group $H_2$, sometimes called \emph{the second Houghton group}, is the group $\FSym(\Z)\rtimes\langle t\rangle$.
\end{defn}
A first important observation is that conjugacy in this group, being a subgroup of $\Sym(\Z)$, is easy to understand: conjugation by $g\in \Sym(\Z)$ changes the support exactly by how it acts on $\Z$ e.g.\ $g^{-1}(a\;b\;c)g=((a)g\;(b)g\;(c)g)$. A commonly used generating set is then $\{(0\;1), t\}$. From \cite{sapir}, $H_2$ has exponential conjugacy growth.

To add clarity to our arguments, we also remark here that elements in $H_2$ do indeed `eventually act as a translation', meaning that their infinite orbits have a simple structure. The following is based on \cite[Prop.\ 2.6]{ConjHou}.
\begin{lem}\label{orbits} Let $g\in H_2$ and $z\in \Z$. If $\{(z)g^m\;:\;m\in\Z\}$ is infinite, then it differs only on finitely many points from a set of the form $$\{x\in \Z\setminus\N\;:\;x\equiv r_1 \Mod{d_1}\}\cup\{x\in\N\;:\;x\equiv r_2 \Mod{d_2}\}$$ for some $d_1, d_2, r_1, r_2\in\N$.
\end{lem}
\begin{proof} Note that $g=\sigma t^c$ for some $\sigma \in \FSym(\Z)$ and $c\in \Z$. Thus, if $x\not\in\supp(\sigma)$, then the result follows by considering $\{(x)g^m\;:\;m\in\N\}$ and $\{(x)g^{-m}\;:\;m\in\N\}$. Then $d_1$ and $d_2$ are both equal to $c$ in this case. Now, if $y\in \supp(\sigma)$, either $y$ lies on a finite orbit or one of the infinite orbits already described.
\end{proof}
\begin{rem}\label{orbitobs} Lemma \ref{orbits} means that for each $g\in H_2$ there are numbers $k\in \Z$ and $z\in \N$ such that $(x)g=x+k$ for every $x\in \Z\setminus\{-z, \ldots, z\}$.
\end{rem}

\section{On a family of two generated subgroups of $H_2$}
The groups we will use were introduced in \cite{Hou2} and are all finite index in $H_2$.
\begin{defn} \label{thegroups} For each $k \in \N$, let $G_k:=\langle \Alt(\Z), t^k\rangle$.
\end{defn}
It follows from \cite[Lem. 5.9]{Cox1} and \cite[Lem. 2.4]{Cox2} that $\Aut(G_k)\cong H_{2k}\rtimes(S_k\wr C_2)$, where the elements in $S_k\wr C_2$ relate to isometrically permuting the $2k$ rays of $X_{2k}$. This means that $G_m\cong G_n$ only if $m=n$.
\begin{rem} Since $\Alt(\Z)$ is an infinite simple group, $G_k$ is not residually finite.
\end{rem}
The following allows us to see that the groups $G_k$ have the property on proper quotients that we desire.
\begin{lem} \cite[Prop. 2.5]{Cox2} Let $\Alt(\Z)\le G\le \Sym(\Z)$. Then $G$ has $\Alt(\Z)$ as a unique minimal normal subgroup.
\end{lem}
\begin{lem} \label{cyclicquotients} Let $k\in \N$. Then every proper quotient of $G_k$ is cyclic.
\end{lem}
\begin{proof}
Let $N$ be a non-trivial normal subgroup of $G_k$. Then $\Alt(\Z)\le N$, and so
\[\quotient{G_k}{N}\cong\quotient{\left(\quotient{G_k}{\Alt(\Z)}\right)}{\left(\quotient{N}{\Alt(\Z)}\right)}\]
i.e.\ $\quotient{G_k}{N}$ is a quotient of $\quotient{G_k}{\Alt(\Z)}\cong \langle t^k\rangle$, and hence cyclic.
\end{proof}
We now show, for each $k\in\N$, that the group $G_k$ is $2$-generated.
\begin{rem} Computations show that $G_1=\langle (1\;2\;3), t\rangle$ and  $G_2=\langle (1\;2\;3), t^2\rangle$.
\end{rem}
The following result is essentially \cite[Lem. 3.2]{invgenhou}.
\begin{not*} For any $k\in \N$, let $\Omega_k=\{1, 2, \ldots, k\}$.
\end{not*}
\begin{lem} Let $k\in \{3, 4, \ldots\}$. Then $G_k$ is generated by $\Alt(\Omega_{2k}) \cup\{t^k\}$.
\end{lem}
\begin{proof} We will show that all $3$-cycles in $\Alt(\Z)$ lie in $\langle \Alt(\Omega_{2k}), t^k\rangle$. We only need to construct those $3$-cycles with support in $\N$, since for any $\sigma \in \Alt(\Z)$ there is an $n \in \N$ such that $\supp(t^{-kn}\sigma t^{kn})\subset \N$. Moreover, $\{(1\;2\;m)\;:\;m\ge 3\}$ generates $\Alt(\N)$ since for any distinct $m_1, m_2, m_3\in \N$ we have $(1\;2\;m_1)(1\;2\;m_2)^{-1}=(2\;m_1\;m_2)$ and $(1\;2\;m_3)^{-1}(2\;m_1\;m_2)(1\;2\;m_3)=(m_1\;m_2\;m_3)$.

By definition, $(1\;2\;m)\in \Alt(\Omega_{2k})$ for $m=3, \ldots, 2k$. To use induction on $n$, let us assume, for each $r\in \{1, \ldots, k\}$, that $(1, 2, nk+r)\in \langle \Alt(\Omega_{2k}), t^k\rangle$. Thus we want to show that $\langle \Alt(\Omega_{2k}), t^k\rangle$ also contains $(1, 2, (n+1)k+r)$. Clearly $(r, r+1, r+k)\in \langle\Omega_{2k}, t^k\rangle$. Let
$$\omega:=t^{-nk}(r, r+1, r+k)t^{nk}=(nk+r, nk+r+1, (n+1)k+r).$$
Conjugating $(1, 2, nk+r)$ by $\omega^{-1}$ yields the required element $(1, 2, (n+1)k+r)$.
\end{proof}

\begin{lem}\label{2gen} Let $k \in \{3, 4, \ldots\}$. Then $G_k$ is 2-generated.
\end{lem}
\begin{proof}
Our aim is to reduce to the previous lemma. From the assumption that $k\ge 3$, all $3$-cycles in $\Alt(\Omega_{2k})$ are conjugate. Our generating set will consist of $t^k$ and an element $\alpha\in \Alt(\Z)$.

Let $n=\binom{2k}{3}$ and let $\omega_1, \ldots, \omega_n$ be a choice of distinct $3$-cycles in $\Alt(\Omega_{2k})$ with $\omega_i\ne\omega_j^{-1}$ for every $i,j \in \{1, \ldots, n\}$. Thus $\langle \omega_1, \ldots, \omega_n\rangle=\Alt(\Omega_{2k})$. Set $\sigma_0=(1\;3)$ and $\sigma_{n+1}=(2\;3)$ and observe that $(\sigma_{n+1}^{-1}\sigma_0^{-1}\sigma_{n+1})\sigma_0=(1\;2\;3)$. Now choose $\sigma_1, \ldots, \sigma_{n}\in \Alt(\Omega_{2k})$ such that for each $m\in\{1,\ldots, n\}$  we have $\sigma_m^{-1}(1\;2\;3)\sigma_m=\omega_m$.

Let $\alpha:=\prod_{i=0}^{n+1}t^{-2ik}\sigma_it^{2ik}$ and, for each $m\in\{1,\ldots, n+1\}$, let $\beta_m:=t^{2mk}\alpha t^{-2mk}$. Then $\beta_{n+1}^{-1}\alpha^{-1}\beta_{n+1}\alpha=\sigma_{n+1}^{-1}\sigma_1^{-1}\sigma_{n+1}\sigma_1=(1\;2\;3)$ and, for $m\in\{1, \ldots, n\}$, we have that $\beta_m^{-1}(1\;2\;3)\beta_m=\omega_m$. Hence $\langle \alpha, t^k\rangle=G_k$.
\end{proof}
\begin{rem} By including an odd element in the choice of the $\sigma_1, \ldots, \sigma_n$, we note that there exists $\alpha_k\in\FSym(\Z)$ such that $\langle \FSym(\Z), t^k\rangle=\langle \alpha_k, t^k\rangle$ for every $k \in \{3, 4, \ldots\}$. Then $\langle \FSym(\Z), t^2\rangle=\langle \alpha_4, t^2\rangle$ since $\FSym(\Z)\le \langle \alpha_4, t^4\rangle\le \langle \alpha_4, t^2\rangle$.
\end{rem}

\section{On the spread of the groups $G_k$, where $k \in \N$}
Now that we have defined our 2-generated groups, we can investigate their spread. Our first aim is to prove the following. Part (i) is Lemma \ref{2gen} and part (ii) is Lemma \ref{cyclicquotients}.
\begin{thm} \label{question1} The groups $G_k:=\langle \Alt(\Z), t^k\rangle$, indexed over the naturals, satisfy:
\begin{enumerate}[i)]
\item for every $k \in \N$, the group $G_k$ is 2-generated;
\item for every $k \in \N$, every proper quotient of $G_k$ is cyclic; but
\item for every $k \in \{3, 4, \ldots\}$, the group $G_k$ is not $\frac32$-generated i.e.\ $s(G_k)=0$.  
\end{enumerate}
\end{thm}
We show (iii) by dealing separately with the cases $k\ge4$ and $k=3$.
\begin{lem}  Let $k\in \{4, 5, \ldots\}$. Then $G_k$ has spread zero.
\end{lem}
\begin{proof} We will show, given any $h \in G_k$, that $\langle (1\;2\;3), h\rangle$ does not contain $\Alt(\Z)$ and so cannot equal $G_k$.

If $h\in \Alt(\Z)$, then $\langle (1\;2\;3), h\rangle$ is finite. Thus, without loss of generality, $h=\omega t^n$ where $\omega \in \Alt(\Z)$ and $n \in k\N$. From our description of the orbits of $H_2$, we know that $h$ has at least $k$ infinite orbits. Thus $h$ has an infinite orbit $\mathcal{O}_4$ which does not intersect $\{1, 2, 3\}$. Let $x\in \mathcal{O}_4$. Then the orbit of $x$ under $\langle (1\;2\;3), h\rangle$ is $\mathcal{O}_4$.
\end{proof}

\begin{lem} The group $G_3$ has spread zero.
\end{lem}
\begin{proof} With a similar approach to the previous lemma, we will show, given any $h \in G_3$, that $\langle (1\;2\;3), h\rangle\ne G_3$.

If $h\in \Alt(\Z)$, then $\langle (1\;2\;3), h\rangle$ is finite. Similarly, if $h=\omega t^n$ where $\omega\in \Alt(\Z)$ and $n\not\in\{3, -3\}$, then $|n|>3$ and the argument in the previous lemma proves the result. Thus, without loss of generality, $h=\omega t^3$ where $\omega \in \Alt(\Z)$.

Let $\mathcal{O}_1, \mathcal{O}_2$, and $\mathcal{O}_3$ denote the three infinite orbits of $h$. After relabelling, we can assume that $i \in \mathcal{O}_i$ for $i=1, 2, 3$ (for if, say, $\mathcal{O}_3$ did not contain a point from $\{1, 2, 3\}$, then we could run the argument from the previous lemma). Moreover, if $\Z\setminus(\mathcal{O}_1\cup \mathcal{O}_2\cup \mathcal{O}_3)=Y\ne \emptyset$, then the orbit of points in $Y$ under $\langle (1\;2\;3), h\rangle$ is the same as the orbit of points in $Y$ under $\langle h\rangle$. By construction this orbit cannot be $\Z$, which contradicts that $\Alt(\Z)\le G_3$ acts transitively on $\Z$. Then $\mathcal{O}_1\cup \mathcal{O}_2\cup \mathcal{O}_3=\Z$, and by relabelling $\Z$ we can assume that $$\mathcal{O}_i=\{n \in \Z\;:\;n\equiv i\mod3\}\text{ for }i=1, 2, 3$$
and $(1\;2\;3)$ has been relabelled as $\sigma=(a_1\;a_2\;a_3)$ where $a_i \in \mathcal{O}_i$ for $i=1, 2, 3$. But then $\langle (a_1\;a_2\;a_3), h\rangle\cong C_3\wr \Z$, since for every $m\in\Z\setminus\{0\}$ we have that $h^{-m}(a_1\;a_2\;a_3)h^m$ and $(a_1\;a_2\;a_3)$ are disjoint. Therefore $\langle (a_1\;a_2\;a_3), h\rangle$ has an imprimitive action on $\Z$, and so cannot contain $\Alt(\Z)$.
\end{proof}

\begin{not*} For any $\omega \in \FSym(\Z)$, let $s_\omega:=\omega^{-1}t\omega$.
\end{not*}
The following intriguing observation tells us that the groups $G_k$ are, in some sense, not far from having infinite spread.
\begin{thm}\label{almostinfinitespread} Let $k, n \in \N$ and $g_1, \ldots, g_n \in G_k\setminus \Alt(\Z)$. Then there exists an $h\in [t^k]\subset G_k$ such that $\langle g_1, h\rangle=\ldots=\langle g_n, h\rangle=G_k$.
\end{thm}
\begin{proof} Fix a $k \in \N$. We start by taking powers of the $g_1, \ldots, g_n$ so to obtain
$$\omega_1t^{-c}, \ldots, \omega_nt^{-c}$$
for some $c\in k\N$ and $\omega_1, \ldots, \omega_n \in \Alt(\Z)$. Our aim is to show that there exists a $\sigma \in \Alt(\Z)$ so that $s_{\sigma}^k$ is a suitable choice for our element $h\in G_k$.

From Remark \ref{orbitobs}, there exists an $m\in \N$ such that for every $z\in \Z\setminus\{-m, \ldots, m\}$ and every $i\in \{1, \ldots, n\}$ we have $\omega_it^{-c}(z)=z-c$. Fix an $i\in \{1, \ldots, n\}$. Let $\omega_i^*:=(\omega_it^{-c})^2t^{2c}$ so that $(\omega_it^{-c})^2=\omega_i^*t^{-2c}$. Then, for any $\sigma \in \Alt(\Z)$,
$$w_i^*t^{-2c}s_\sigma^{2c}=\omega_i^*(t^{-2c}\sigma^{-1}t^{2c})\sigma$$
and we can choose $\sigma$ so that $\supp(\sigma)\cap \supp(t^{-2c}\sigma^{-1}t^{2c})=\emptyset$. Recall that $\Omega_{2c}=\{1, \ldots, 2c\}$. We will choose $\sigma$ so that $\supp(\sigma)\cap \{-m, \ldots, m\}=\emptyset$ and
$$\supp(\sigma)\subset \bigcup_{j\in \Z}\Omega_{2c}t^{4jc}.$$

We now mimick the proof of Lemma \ref{2gen} (with $\omega_it^{-c}$ taking the role of $t^k$) and show that there exists $\sigma \in \FSym(\Z)$, with the conditions above, such that $\langle \omega_i^*(t^{-2c}\sigma^{-1}t^{2c})\sigma, \omega_it^{-c}\rangle$ contains $\Lambda:=\{\alpha \in \Alt(\Z)\;:\;\min(\supp(\alpha))\ge m\}$. Let $n=\binom{2c}{3}$, pick a set of distinct $3$-cycles $\tau_1, \ldots, \tau_n$ with support in $\Omega_{2c}$ such that $\tau_i\ne\tau_j^{-1}$ for any $i, j\in\{1, \ldots, n\}$, and choose $\sigma_1, \ldots, \sigma_n$ with support in $\Omega_{2c}$ such that $\sigma_i^{-1}(1\;2\;3)\sigma_i=\tau_i$ for $i=1, \ldots, n$. Choose $\sigma$ so that $t^{4c(m+i)}\sigma t^{-4c(m+i)}$ acts on $\Omega_{2c}$ in the same way as $\sigma_i$ for $i=1, \ldots, n$, and moreover make $\sigma$ act on $\Omega_{2c}t^{-4cm}$ and $\Omega_{2c}t^{4c(m+n+1)}$ in such a way that there exists $d\in\N$ so that the commutator of $(\omega_it^{-c})^d(\omega_i^*(t^{-2c}\sigma^{-1}t^{2c})\sigma)(\omega_it^{-c})^{-d}$ and $\omega_i^*(t^{-2c}\sigma^{-1}t^{2c})\sigma$ produces $t^{-4c(m+n+1)}(1\;2\;3)t^{4c(m+n+1)}$. We can assume that $\sigma\in \Alt(\Z)$ by potentially composing it with $(x\;x+1)$ for some suitably large $x\in\N$. Then, for every $i\in \{1, \ldots, n\}$, we have that $\langle s_{\sigma}^k, g_i\rangle\ge \langle s_{\sigma}^k, \Lambda\rangle\ge \langle s_\sigma^k, \Alt(\Z)\rangle=G_k$.
\end{proof}
\begin{rem} \label{avoiding} There is great flexibility for the element $\sigma$ in the above proof. First, given any $m'\in \N$, we can insist that $\supp(\sigma)\cap\{-m', \ldots, m'\}=\emptyset$. Secondly, given any $\tau\in\Alt(\Z)$, we can choose $m'$ large enough so that $\supp(\sigma)\cap\supp(\tau)=\emptyset$ and that $s_{\sigma\tau}^k$ and $s_{\tau\sigma}^k$ both generate $G_k$ with each of $g_1, \ldots, g_n\in G_k\setminus\Alt(\Z)$.
\end{rem}
The previous theorem reveals curious properties for the generating graphs of each group $G_k$ (which is a graph with vertex set $G$ and an edge between vertices which generate the group). If we consider the elements of $G_k$ as obtained through the union of balls $B_n(G_k, X)$ for some finite generating set $X$ of $G_k$, then, by the previous lemma, for each $m\in\N$ we can construct an element $\gamma_m\in\Alt(\Z)$ such that $\gamma_m^{-1}t^k\gamma_m$ generates $G_k$ with every element in $B_m(G_k, S)\setminus \Alt(\Z)$. Since $B_n(G_k, X)$ grows exponentially with $n$, the set $\{\gamma_m^{-1}t^k\gamma_m\;:\;m\in\N\}$ can be considered as `small' in $G_k$. Note also the similarity to what is known, from \cite{H}, for finite simple groups. In that world, there is a `small' set of elements $S$ - in certain cases of size 2 - such that for every $g \in G\setminus\{1\}$ there exists an $s \in S$ such that $\langle g, s\rangle=G$.
\section{Conditions for the generation of $G_1$ and $G_2$}
We start with a criteria for when a $3$-cycle generates $G_1$ with the element $t$.
\begin{lem} \label{3cyclecondition} Let $\omega\in \Alt(\Z)$, $s:=s_\omega$, and $a, b\in \Z$ with $a<0<b$. Define $d_a, d_b\in \N$ such that $(a)s^{d_a}=0$ and $(b)s^{-d_b}=0$. Then $\langle (a\;0\;b), s\rangle=G_1$ if and only if $\gcd(d_a, d_b)=1$.
\end{lem}

\begin{proof}
We change our perspective by relabelling $\Z$ so to replace $s$ with $t$ and $(a\;0\;b)$ with $(-d_a\;0\;d_b)$. If $\gcd(d_a, d_b)=1$, then apply the Euclidean algorithm to $(-d_a\;0\;d_b)$ using $t$ to obtain $(-1\;0\;m)$ or $(-m\;0\;1)$ for some $m \in \N$, which both then reduce to $(-1\;0\;1)$.

On the other hand, if $\gcd(d_a, d_b)=k>1$, then we claim that the set of torsion elements in $\langle (-d_a\;0\;d_b), t\rangle$ preserve a non-trivial block structure of $\Z$. We work with $\langle (-k\;0\;k), t\rangle$, since this contains $\langle (-d_a\;0\;d_b), t\rangle$. A torsion element in $\langle (-k\;0\;k), t\rangle$ will be a product of $t$-conjugates of $(-k\;0\;k)$ and $(-k\;0\;k)^{-1}$. Then both $(-k\;0\;k)$ and $(-k\;0\;k)^{-1}$ fix all but one of the sets
$$X_i:=\{z\in \Z\;:\;z\equiv i \mod{k}\}\text{ where }i=1, \ldots, k$$
and every $t$-conjugate of $(-k\;0\;k)^{\pm1}$ also has this property.
\end{proof}

The following is a specific case of \cite[Lem. 2.4]{invgenhou}.
\begin{lem} \label{32genG1} Let $\sigma\in \Alt(\Z)$. Then there exists an $\omega\in \FSym(\Z)$  such that $\langle s_\omega, \sigma\rangle=G_1$. Moreover, given any $\tau \in \FSym(\Z)$ with $\supp(\tau)$ disjoint from $\supp(\sigma)\cup\supp(\omega)$, we have that $\langle s_{\tau\omega}, \sigma\rangle=\langle s_{\omega\tau}, \sigma\rangle=G_1$.
\end{lem}
\begin{proof}
As with the previous lemma, we start by reordering $\Z$ to consider $t$ rather than $s_\omega$, at the cost of then working with $\sigma_*=\gamma^{-1}\sigma\gamma$ rather than $\sigma$. With this perspective, the claims regarding $\tau$ follow immediately.

Let $\sigma_*=\sigma_1\ldots\sigma_n$ be written in disjoint cycle notation and let $\max\{\supp(\sigma_*)\}=\max\{\supp(\sigma_n)\}=y$. We can choose $\gamma$ so that $(y)\sigma_*^{-1}=y-1$.

Conjugate $\sigma_*$ by a power of $t$ to obtain $\alpha$ where $\min\{\supp(\alpha)\}=y$. We can then conjugate $\sigma_*$ by $\alpha$ to obtain $\sigma'$ where, by construction, $(x)\sigma_*=(x)\sigma'$ for all $x<y$.

We will now consider the element $\sigma_*(\sigma')^{-1}$. Recall that $\sigma_*=\sigma_1\ldots\sigma_n$, and, from our construction, $\sigma'=\sigma_1\ldots\sigma_{n-1}\sigma_n'$ where
\begin{equation*}
\sigma_n=(x_1^{(n)}\ldots x_{r_n-1}^{(n)}\;x_{r_n}^{(n)})\text{ and }\sigma_n'=(x_1^{(n)}\ldots x_{r_n-1}^{(n)}\;m)\text{ for some }m>x_{r_n}^{(n)}.
\end{equation*}
Therefore $\sigma_*(\sigma')^{-1}=\sigma_n(\sigma'_n)^{-1}=(x_{r_n-1}^{(n)}\;x_{r_n}^{(n)}\;m)=(y-1\;y\;m)$, which then, together with $t$, generates $G_1$ by Lemma \ref{3cyclecondition}.
\end{proof}
\begin{lem} \label{genG2lem} For any $z\in \Z$, we have that $\langle (2z-1\;0\;2), t^2\rangle=\langle (2z\;1\;3), t^2\rangle=G_2$.
\end{lem}
\begin{proof} Starting with an arbitrary odd $d\in \Z$, if we conjugate $(d\;0\;2)$ by $t^{-2}(d\;0\;2)t^2$ then we obtain $(d\;0\;4)$. Then $(d\;0\;2)(d\;4\;0)=(0\;2\;4)$ which conjugates $(d\;0\;2)$ to $(d\;2\;4)$. Now $t^2(d\;2\;4)t^{-2}=(e\;0\;2)$ where $e=d-2$, and hence $\langle (d\;0\;2), t^2\rangle$ contains $(1\;0\;2)$ and so $\Alt(\Z)$. For the other case, we  note that $tG_2t^{-1}=G_2$.
\end{proof}
The preceding lemma allows us to generalise the argument of Lemma \ref{32genG1} to $G_2$.
\begin{lem} \label{32genG2} Let $\sigma\in \Alt(\Z)$. Then there exists an $\omega\in \FSym(\Z)$  such that $\langle s_\omega^2, \sigma\rangle=G_2$. Moreover, given any $\tau \in \FSym(\Z)$ with $\supp(\tau)$ disjoint from $\supp(\sigma)\cup\supp(\omega)$, we have that $\langle s_{\tau\omega}^2, \sigma\rangle=\langle s_{\omega\tau}^2, \sigma\rangle=G_2$.
\end{lem}
\begin{proof} As with the proof of Lemma \ref{32genG1}, we change our perspective and so work with $t^2$ rather than $s_\omega^2$, at the cost of then working with $\sigma_*=\gamma^{-1}\sigma\gamma$ rather than $\sigma$. This again resolves the claims regarding $\tau$.

If $\sigma$ is a $3$-cycle, then we choose $\gamma$ so that $\supp(\sigma_*)=\{0, 1, 2\}$. Otherwise we recall the notation from Lemma \ref{32genG1}, and write $\sigma_*=\sigma_1\ldots\sigma_n$ where
\begin{equation*}
\sigma_i=(x_1^{(i)}\ldots x_{r_i}^{(i)})\text{ for each } i\in\{1, \ldots, n\}
\end{equation*}
and choose $\gamma\in\Alt(\Z)$ so that:
\begin{itemize}
\item $x_{r_n}^{(n)}$ is even and maximal in $\supp(\sigma_*)$;
\item $x_{r_n-1}^{(n)}$ is odd;
\item $x_1^{(1)}=0$ and is minimal in $\supp(\sigma_*)$; and
\item $x_2^{(1)}=2$.
\end{itemize}
Now conjugate $\sigma_*$ by a suitable power of $t^2$ to produce an element $\alpha\in\Alt(\Z)$ with $\min\{\supp(\alpha)\}=x_{r_n}^{(n)}$ and conjugate $\sigma_*$ by $\alpha$ to obtain $\sigma'$. With this construction $\sigma_*(\sigma')^{-1}=(x_{r_n-1}^{(n)}\;x_{r_n}^{(n)}\;x_{r_n}^{(n)}+2)$. Then $\langle \sigma_*(\sigma')^{-1}, t^2\rangle=G_2$ by Lemma \ref{genG2lem}.
\end{proof}

A group $G$ has a dominating set $S$ if for each $g\in G\setminus\{1\}$ there exists an $s\in S$ such that $\langle g, s\rangle=G$. A natural question is whether $G_1$ or $G_2$ has a finite dominating set. Despite the strong result of Theorem \ref{almostinfinitespread}, we find that the groups $G_k$ do not resolve the question of whether there is an infinite group - other than $\Z$ or the Tarski monsters - with a finite total dominating set.
\begin{prop} \label{nofinitedomset} No finite total dominating set exists for $G_1$ or $G_2$. 
\end{prop}
\begin{proof} Consider the elements in $L=\{(k\;k+2\;k+4)\;:\;k\in \N\}$. Given any finite sets $S_1\subset G_1$ and $S_2\subset G_2$, first produce $S_1^*$ and $S_2^*$ such that  $S_i^*$ contains elements with exactly $i$ infinite orbits. We do this because no element in $S_i\setminus S_i^*$ can generate $G_i$ with any element in $L$.

Recall, by Remark \ref{orbitobs}, that there exists an $m\in \N$ such that each $g\in S_1^*\cup S_2^*$ translates all points $z>m$. Now $(m\;m+2\;m+4)$ cannot generate $G_1$ with any $s \in S_1^*$ by Lemma \ref{3cyclecondition}, and $(m\;m+2\;m+4)$ cannot generate $G_2$ with any $s'\in S_2^*$ since $m$, $m+2$, and $m+4$ all lie on the same orbit of each $s'\in S_2^*$.
\end{proof}

\section{Bounds on the spread and uniform spread of $G_1$ and $G_2$}
Our aim in this final section is to prove the following.
\begin{thm} \label{spreadofG1G2} Let $G_1=\langle (1\;2\;3), t\rangle$ and $G_2=\langle (1\;2\;3), t^2\rangle$. Then $u(G_1)\ge 2$, $u(G_2)\ge 2$, $s(G_1)\le 34$, and $s(G_2)\le 9$.
\end{thm}
Our proof naturally breaks into four parts. To show that $u(G_1)\ge 2$, we show that given any $g_1, g_2 \in G_1$ there exists $s\in [t]$ which generates $G_1$ with each of $g_1$ and $g_2$. In a similar way, but using $[t^2]$ rather than $[t]$, we show that $u(G_2)\ge 2$. The upper bounds for $s(G_1)$ and $s(G_2)$ are found in order to illustrate that these groups have finite spread, and so are unlikely to be optimal.
\begin{not*} Given $g, h \in G$, let $g^h:=h^{-1}gh$.
\end{not*}
\begin{lem} \label{genwith2elements} Let $g_1, g_2\in \Alt(\Z)$. Then there exists an $h\in [t]\subset G_1$ such that $\langle h, g_1\rangle=\langle h, g_2\rangle=G_1$.
\end{lem}
\begin{proof} By reordering $\Z$, we can work with $t$ rather than $h$ and aim to find an $\omega\in\Alt(\Z)$ such that $\langle g_1^\omega, t\rangle=\langle g_2^\omega, t\rangle=G_1$.

First assume that $\supp(g_1)=\supp(g_2)$. If $|\supp(g_1)|=3$, then $g_1\in\{g_2, g_2^{-1}\}$ and so we are done. Otherwise, let $y=\max\{\supp(g_1)\}$, and define $x_1:=(y)g_1^{-1}$, $x_2:=(y)g_2^{-1}$, and fix some $a\in \supp(g_1)\setminus\{x_1, x_2, y\}$. Choose $\omega\in \Alt(\Z)$ so that:
\begin{itemize}
\item $(y)\omega=:y'$ is still maximal in $\supp(g_1^\omega)$;
\item $a$ is minimal in $\supp(g_1^\omega)$; and
\item $(x_1)\omega=y'-1$.
\end{itemize}
Applying the steps in the proof of Lemma \ref{32genG1} to $g_1^\omega$ then yields $(y'-1\;y'\;m)$ for some $m>y'$, and so Lemma \ref{3cyclecondition} implies that $\langle g_1^\omega, t\rangle=G_1$. We note that there is also some $d\in \N$ such that $(y')t^{-d}=(x_2)\omega$. Given $b_2:=(a)g_2^\omega$, we can therefore adjust $\omega$ so that it moves $a$ arbitrarily far away from the other points in $\supp(g_1^\omega)$ i.e\ we can assume that $(a)t^p=b_2$, where $p$ is a suitably large prime number so that $\gcd(d, p)=1$. Again, applying the steps from the proof of Lemma \ref{32genG1} will yield $(y'-d\;y'\;y'+p)$ and so $\langle g_2^\omega, t\rangle\ge\langle (-d\;0\;p), t\rangle=G_1$ by Lemma \ref{3cyclecondition}.

If $\supp(g_1)\ne\supp(g_2)$ then, without loss of generality, $ \supp(g_1)\setminus \supp(g_2)\ne\emptyset$ and we can fix some $\alpha$ in this set. Imagine that $\{\alpha g_1^k\;:\;k\in \N\}\cap \supp(g_2)=\emptyset$. By choosing $\omega\in\FSym(\Z)$ such that $(\alpha)g_1^\omega=\alpha+1$ and $\alpha=\min(\supp(g_1^\omega)\cup\supp(g_2^\omega))$, we have that $\langle t, g_1^\omega\rangle=G_1$. Thus choose $\omega$ so that some point $y$ in $\supp(g_2^\omega)$ is maximal in $\supp(g_1^\omega)\cup\supp(g_2^\omega)$, and set $(y)(g_2^\omega)^{-1}=y-1$.

Now let $\alpha\in \supp(g_1)\setminus \supp(g_2)$ but assume that $(\alpha)g_1^\omega=\beta\in \supp(g_2^\omega)$ where $\alpha$ is minimal in $\supp(g_1^\omega)\cup\supp(g_2^\omega)$ and $\beta=\alpha+1$. Thus $\langle t, g_1^\omega\rangle=G_1$ but also, by construction, $\beta$ is minimal in $\supp(g_2^\omega)$. We can therefore choose $\omega$ so that $(\beta)g_2^\omega=\beta+1$, so that $\langle t, g_2^\omega\rangle=G_1$.
\end{proof}

\begin{lem} \label{spreadG1lb} The group $G_1$ has uniform spread at least 2.
\end{lem}
\begin{proof} Let $g_1, g_2 \in G_1$. We will show that there exists $\alpha \in \FSym(\Z)$ such that $\langle s_\alpha, g_1\rangle = \langle s_\alpha, g_2\rangle=G_1$. The case where $g_1, g_2 \in \Alt(\Z)$ is dealt with in the previous lemma. If $g_1, g_2 \in G_1\setminus\Alt(\Z)$, then Theorem \ref{almostinfinitespread} implies the result. Finally, if $g_1 \in \Alt(\Z)$ and $g_2\in G_1\setminus\Alt(\Z)$, then we can apply Lemma \ref{32genG1} so to find $s_\omega$ which generates $G_1$ with $g_1$. This lemma also states that choosing any $\tau\in \FSym(\Z)$ such that $\supp(\tau)$ is disjoint from $\supp(\omega)\cup\supp(g_1)$ will result in $\langle s_{\tau\omega}, g_1\rangle=G_1$. Theorem \ref{almostinfinitespread} and Remark \ref{avoiding} together ensure the existence of an element $\alpha=\sigma\omega\in\Alt(\Z)$ so that $\langle g_1, s_\alpha\rangle=\langle g_2, s_\alpha\rangle=G_1$.
\end{proof}
We now show that $s(G_2)\ge 2$. As we did for $G_1$, we deal separately with the case where the given elements are both in $\Alt(\Z)$. We do approach $G_2$ differently however, since the methods we used for $G_1$ appear to be too weak. This is because the conditions for a $3$-cycle to generate $G_2$ with $t^2$ are more restrictive than the conditions for a $3$-cycle to generate $G_1$ with $t$. Our approach takes some inspiration from \cite[Cor. 3.04]{spread}, by partitioning $\Z$ into two parts and working separately with each. We partition $\Z$ into even and odd numbers, which we use because these are the two orbits of $t^2$.
\begin{not*} Let $2\Z$ and $2\Z+1$ denote the sets of even and odd integers, respectively.
\end{not*}
Our first step is to show that we can focus on generating $\Alt(2\Z+1)$ rather than $\Alt(\Z)$. We will then use this simplification in order to show that $u(G_2)\ge2$.
\begin{lem} \label{evensandodds} If $\sigma\in \Alt(\Z)$ has an orbit containing both an even and an odd integer, then $\langle \sigma, \Alt(2\Z+1), t^2\rangle =G_2$.
\end{lem}
\begin{proof} First, conjugate $\sigma$ by a suitable power of $t^2$ to produce $\sigma_*$, where $\sigma_*$ sends $0$ to a point in $2\Z+1$ and $\sigma_*$ sends no negative even number to $2\Z+1$. Now, given $\omega \in \Alt(\Z)$ with $|\supp(\omega)\cap2\Z|=k>0$, we consider the following steps:
\begin{enumerate}[i.]
\item conjugate $\omega$ by a suitable power of $t^2$ to produce $\omega'$ so that $0\in\supp(\omega')$ and $\supp(\omega')\cap 2\N=\emptyset$;
\item conjugate $\omega'$ by an appropriate element of $\Alt(2\Z+1)$ to produce $\omega_*$ with $\supp(\sigma_*)\cap\supp(\omega_*)\subset 2\Z$;
\item conjugate $\omega_*$ by $\sigma_*$ to produce $\hat{\omega}$.
\end{enumerate}
Note that $|\supp(\hat{\omega})\cap2\Z|=|\supp(\omega)\cap2\Z|-1$. Iterating these steps therefore results in an element in $\Alt(2\Z+1)$.
\end{proof}

\begin{prop} Given any $g_1, g_2\in \Alt(\Z)$, there exists $s\in [t^2]$ which generates $G_2$ with both $g_1$ and $g_2$.
\end{prop}
\begin{proof} Again, by relabelling $\Z$, we deal with the equivalent problem of whether there exists $\omega \in \Alt(\Z)$ such that $\langle g_1^{\omega}, t^2\rangle =\langle g_2^{\omega}, t^2\rangle=G_2$. If $\supp(g_1)\cap\supp(g_2)=\emptyset$, then Lemma \ref{32genG2} implies the result. Otherwise, fix some $c\in \supp(g_1)\cap\supp(g_2)$. We will restrict ourselves to those $\omega\in \Alt(\Z)$ such that:
\begin{itemize}
\item $(c)\omega=0$;
\item $(\supp(g_1^{\omega})\cup\supp(g_2^{\omega}))\cap2\Z=\{0\}$; and
\item the only positive odd integers in $\supp(g_1^{\omega})\cup\supp(g_2^{\omega})$ are $1$ and $3$.
\end{itemize}
In light of the previous lemma, our aim is to choose $\omega$ so to relabel the points in $\supp(g_1)\cup\supp(g_2)$ so that both $\langle g_1^{\omega}, t^2\rangle$ and $\langle g_1^{\omega}, t^2\rangle$ contain $\Alt(2\Z+1)$.

Note that $|\supp(g_1)|\ge 3$. Hence there exist distinct $x, y$, both not equal to $c$, with $(x)g_1=y$. We now further restrict our choice of $\omega$ to those such that:
\begin{itemize}
\item $(x)\omega=1$; and
\item $(y)\omega=3$.
\end{itemize}
To simplify the cases we consider later, we impose that if $y\in \supp(g_2)$, then so is $x$. This is not a problematic choice: if only $y$ were in $\supp(g_2)$, then we could replace $g_1$ with $g_1^{-1}$ and then swap the labels for $x$ and $y$. We first note that any choice of $\omega\in \Alt(\Z)$ with the above properties gives $\Alt(2\Z+1)\le \langle  g_1^{\omega}, t^2\rangle$. We prove this in two parts. If $|\supp(g_1)|=3$, then we apply Lemma \ref{genG2lem} to obtain that $\langle  g_1^{\omega}, t^2\rangle=G_2$. Otherwise we apply the following steps, which we have employed numerous times in this paper so far (the first time being in Lemma \ref{32genG1}).
\begin{enumerate}[1.]
\item Let $a:=\min(\supp(g_1^\omega)\setminus\{0, 1, 3\})$. From our restrictions on $\omega$, $a$ is odd.
\item Conjugate $g_1^\omega$ by a power of $t^2$ to obtain $\alpha$, an element with $\min(\supp(\alpha))=3$.
\item Conjugate $g_1^\omega$ by $\alpha$ to obtain $g_1'$. Note, for every $z \in\Z\setminus\{1, 3, (3)g_1^\omega\}$, that $(z)g_1^\omega=(z)g_1'$.
\item Let $g_1^*:=g_1^\omega (g_1')^{-1}$. Then $g_1^*=(1\;3\;m)$, where $m=(3)\alpha$.
\end{enumerate}
This element $g_1^*$ is exactly what we need. If $(3)\alpha\in 2\Z$ then $\langle g_1^*, t^2\rangle =G_2$ by Lemma \ref{genG2lem}. If $(3)\alpha\in 2\Z+1$ then, by focusing on the action of $t^2$ and $(1\;3\;m)$ on $2\Z+1$, we can apply Lemma \ref{3cyclecondition} to see  that $\Alt(2\Z+1)\le \langle g_1^*, t^2\rangle$.

We must now choose the image of $\omega$ on $\supp(g_2)\setminus\{c, x, y\}$ so to guarantee that $\Alt(2\Z+1)\le \langle  g_2^{\omega}, t^2\rangle$. First, if $1$ and $3$ lie on the same orbit of $g_2^{\omega}$, then there exists $f_2$, a suitable power of $g_2^\omega$, such that $(1)f_2=3$. The steps given above then imply that $\langle f_2, t^2\rangle=G_2$ for any choice of $\omega$ with the properties we have specified above. Similarly if there exist $a, b\in \supp(g_2)\setminus\{c, x, y\}$ which lie on the same orbit of $g_2$, then we can produce $h_2$, a suitable power of $g_2$ such that $(a)h_2=b$. We can insist that $\omega$ preserves the ordering on $\supp(h_2)\setminus\{c, x, y, a, b\}$, and let $q:=\max(\supp(h_2))$. By choosing $\omega$ so that:
\begin{itemize}
\item $a$ is odd;
\item $a$ is minimal in $\supp(g_2^\omega)$; and
\item $b=a+2$
\end{itemize}
we see that the steps (1)-(4) above yield the $3$-cycle $((1)h_2^{-1}\;q\;q+2)$. As we saw for $g_1^*$, Lemma \ref{genG2lem} and Lemma \ref{3cyclecondition} deal respectively with the cases that $(1)h_2^{-1}$ is even or odd. We now note that if $1$ and $3$ lie on distinct $g_2$ orbits and there are no such $a, b\in \supp(g_2)\setminus\{c, x, y\}$ which lie on the same orbit of $g_2$, then we have one of a finite list of cases. We end by dealing with each such possibility.

If $x, y\not\in\supp(g_2)$, then $|\supp(g_2)|\ge 3$ and so there exist $a, b\in \supp(g_2)\setminus\{c\}$ that lie on the same orbit of $g_2$. If $x\in\supp(g_2)$ but $y\not\in\supp(g_2)$, then either:
\begin{itemize}
\item $g_2^\omega$ or its inverse equals $(i\;0\;1)$ for some $i\in\Z\setminus\{0, 1\}$, in which case, by choosing the appropriate $\omega$, we can set $i:=-1$ so that $\langle g_2^\omega, t^2\rangle=G_2$; or
\item $g_2^\omega=(i\;0)(j\;1)$ for some distinct $i, j\in \Z\setminus\{0, 1\}$. We set $i:=-5$ and $j:=-1$ and note that $\langle g_2^\omega, t^2\rangle$ contains $g_2^\omega((t^{-2}g_2^\omega t^2)^{-1}g_2^\omega(t^{-2}g_2^\omega t^2))=(-1\;1\;3)$. We can then apply Lemma \ref{3cyclecondition} to the orbit of $t^2$ containing $2\Z+1$ in order to generate $\Alt(2\Z+1)$, and then apply Lemma \ref{evensandodds} for our desired conclusion that $\langle g_2^\omega, t^2\rangle=G_2$.
\end{itemize}
We commented above that we could assume that if $y\in \supp(g_2)$, then $x$ must also be in $\supp(g_2)$. So finally, if $x, y \in \supp(g_2)$, then $g_2^\omega$ or its inverse is either $(i\;1\;0)(j\;3)$, $(i\;0\;3)(j\;1)$, or $(i\;0)(j\;1)(k\;3)$. But then $g_2\not\in\Alt(\Z)$, since these are all odd permutations. 
\end{proof}
We are now equipped to show that $u(G_2)\ge 2$, which we approach in a similar way to the proof of Lemma \ref{spreadG1lb}.
\begin{lem} The group $G_2$ has uniform spread at least 2.
\end{lem}
\begin{proof} Let $g_1, g_2 \in G_1$. We will show that there exists $\alpha \in \FSym(\Z)$ such that $\langle s_\alpha^2, g_1\rangle = \langle s_\alpha^2, g_2\rangle=G_2$. If $g_1, g_2 \in G_2\setminus\Alt(\Z)$, then Theorem \ref{almostinfinitespread} implies the result. If $g_1 \in \Alt(\Z)$ and $g_2\in G_2\setminus\Alt(\Z)$, then we can apply Lemma \ref{32genG2} in the same way to how Lemma \ref{32genG1} is used in the proof of Lemma \ref{spreadG1lb}. The final case was dealt with in the previous proposition.
\end{proof}
In light of Theorem  \ref{almostinfinitespread}, we consider a finite collection of elements in $\Alt(\Z)$ to obtain upper bounds on the spread of $G_1$ and $G_2$. I thank François Le Maître for his improvements to the following argument.
\begin{lem}\label{lemG1ub} The group $G_1$ has spread at most 9.
\end{lem}
\begin{proof} We will construct a set $S\subset G_1$, of size 10, such that no $h\in G_1$ generates $G_1$ with every element in $S$. The set will consist of $3$-cycles in $\Alt(\Z)$ with support contained within $X=\{1, \ldots, 5\}$ and chosen such that if $\sigma\in S$ then $\sigma^{-1}\not\in S$. There are $\binom{5}{3}=10$ such cycles. Let $h\in G_1$. Then $h=\sigma t^k$ for some $\sigma \in \Alt(\Z)$ and $k\in \Z$. If $k\not\in\{1, -1\}$, then $h$ cannot generate $G_1$ with $(1\;2\;3)$.

We start with the observation that, given 5 points in $\Z$, they cannot all lie on the infinite orbit of $h$. To do this we use Lemma \ref{3cyclecondition}. Assume that $h$ consists of a single infinite orbit, equal to $\Z$. Without loss of generality, reorder $\Z$ so to replace $h$ with $t$, the translation map. Now, given $\{a, b, c, d, e\}\subset\Z$ with $a<b<c<d<e$, at least 3 of these points lie in $2\Z$ or in $2\Z+1$. Let $\sigma$ denote the 3-cycle with support equal to 3 such points. Lemma \ref{3cyclecondition} then states that $\langle \sigma, t\rangle\ne G_1$.

From the previous paragraph, no $h$ with an infinite orbit equal to $\Z$ can generate $G_1$ with each element from $S$. Hence at least one of $\{1, \ldots, 5\}$ must lie on a finite orbit of $h$. Let $K$ be the intersection of the infinite orbit of $h$ with $\{1, \ldots, 5\}$, and let $L:=\{1, \ldots, 5\}\setminus K$. One of $K$ or $L$ must have size 3, and so contains $m_1<m_2<m_3$ from $\{1, \ldots, 5\}$. Let $\omega=(m_1\;m_2\;m_3)$. Then $\langle h, \omega\rangle\ne G_1$ since the group generated by $h$ and $\omega$ does not act transitively on $\Z$.
\end{proof}

\begin{lem} The group $G_2$ has spread at most 9.
\end{lem}
\begin{proof} As with the previous proof, let $S$ consist of those $3$-cycles in $\Alt(\Z)$ with support contained within $X=\{1, \ldots, 5\}$ and chosen such that if $\sigma\in S$ then $\sigma^{-1}\not\in S$. We will show that no $h\in G_2$ generates $G_2$ with every element in $S$. Given $h\in G_2$, we have that $h=\sigma t^{2k}$ where $\sigma \in \Alt(\Z)$ and $k \in \Z$. If $k\not\in\{1, -1\}$, then $h$ cannot generate $G_2$ with $(1\;2\;3)$. Let $h$ have infinite orbits $\mathcal{O}_1$ and $\mathcal{O}_2$ and define $\phi_h: X\rightarrow \{0, 1, 2\}$ by
\begin{itemize}
\item $0$ if $x$ lies on a finite orbit of $h$;
\item $1$ if $x\in\mathcal{O}_1$; and
\item $2$ if $x\in\mathcal{O}_2$.
\end{itemize}
It is clear that the image of $\phi$ must contain $\{1, 2\}$, since otherwise $\langle h, (a\;b\;c)\rangle$ cannot act transitively on both $\mathcal{O}_1$ and $\mathcal{O}_2$ for any distinct $a, b, c\in X$. Similarly, we cannot have two points of $X$ sent to 0 by $\phi$, and we cannot have 3 points of $X$ sent to either 1 or to 2. Thus, after relabelling, the only possibility is that $\phi(1)=0$, $\phi(2)=\phi(3)=1$, and $\phi(4)=\phi(5)=2$. But then $\langle h, (1\;2\;3)\rangle\ne G_2$. 
\end{proof}
\newpage
\bibliographystyle{amsalpha}
\def\cprime{$'$}
\providecommand{\bysame}{\leavevmode\hbox to3em{\hrulefill}\thinspace}
\providecommand{\MR}{\relax\ifhmode\unskip\space\fi MR }
\providecommand{\MRhref}[2]{%
 \href{http://www.ams.org/mathscinet-getitem?mr=#1}{#2}
}
\providecommand{\href}[2]{#2}

\end{document}